\documentclass[a4paper, 11pt]{amsart} 

\usepackage[letterpaper,margin=1in]{geometry}
\usepackage{etex,float}
\usepackage{amsmath,amssymb,amsthm,amsfonts,mathrsfs, mathtools}
\usepackage[frame,cmtip,arrow,matrix,line,graph,curve]{xy}
\usepackage{graphpap,color,paralist,pstricks}
\usepackage[mathscr]{eucal}
\usepackage{mathabx}
\usepackage[pdftex,colorlinks,backref=page,citecolor=blue]{hyperref}
\usepackage{tikz}
\usetikzlibrary{calc,decorations.markings}
\usepackage{epic,eepic}
\usepackage{yfonts}
\usepackage{enumitem} 
\usepackage{bbm}
\usepackage{tikz-cd}
\usepackage{aliascnt}
\usepackage{mathbbol}

\allowdisplaybreaks

\usepackage{xcolor, color, soul}

\usepackage{color}

\newtheorem{theorem}{Theorem}[section]

\newaliascnt{headcor}{headthm}

\aliascntresetthe{headcor}

\newaliascnt{headconj}{headthm}

\aliascntresetthe{headconj}

\newaliascnt{corollary}{theorem}
\newtheorem{corollary}[corollary]{Corollary}
\aliascntresetthe{corollary}

\newaliascnt{claim}{theorem}

\aliascntresetthe{claim}

\newaliascnt{lemma}{theorem}
\newtheorem{lemma}[lemma]{Lemma}
\aliascntresetthe{lemma}

\newaliascnt{thmdfn}{theorem}
\newtheorem{thmdfn}[thmdfn]{Theorem/Definition}
\aliascntresetthe{thmdfn}

\newaliascnt{conjecture}{theorem}
\newtheorem{conjecture}[conjecture]{Conjecture}
\aliascntresetthe{conjecture}

\newaliascnt{proposition}{theorem}
\newtheorem{proposition}[proposition]{Proposition}
\aliascntresetthe{proposition}

\theoremstyle{definition}
\newaliascnt{definition}{theorem}
\newtheorem{definition}[definition]{Definition}
\aliascntresetthe{definition}

\newaliascnt{notation}{theorem}
\newtheorem{notation}[notation]{Notation}
\aliascntresetthe{notation}

\newaliascnt{example}{theorem}
\newtheorem{example}[example]{Example}
\aliascntresetthe{example}

\newaliascnt{examples}{theorem}

\aliascntresetthe{examples}

\newaliascnt{remark}{theorem}
\newtheorem{remark}[remark]{Remark}
\aliascntresetthe{remark}

\newaliascnt{fact}{theorem}

\aliascntresetthe{fact}

\newaliascnt{question}{theorem}

\aliascntresetthe{question}

\newaliascnt{questions}{theorem}

\aliascntresetthe{questions}

\newaliascnt{problem}{theorem}

\aliascntresetthe{problem}

\newaliascnt{construction}{theorem}

\aliascntresetthe{construction}

\newaliascnt{setup}{theorem}

\aliascntresetthe{setup}

\newaliascnt{algorithm}{theorem}

\aliascntresetthe{algorithm}

\newaliascnt{observation}{theorem}

\aliascntresetthe{observation}

\newaliascnt{discussion}{theorem}

\aliascntresetthe{discussion}

\newaliascnt{defprop}{theorem}

\aliascntresetthe{defprop}

\def\sectionautorefname~#1\null{Section #1\null}
\def\subsectionautorefname~#1\null{\S #1\null}

\makeatletter
\def\equationautorefname~#1\null{Equation\,(#1)\null}
\makeatother

\definecolor{myorange}{RGB}{255, 160, 70}
\definecolor{darkgreen}{RGB}{0, 166, 0}

\def \pol{{\operatorname{pol}}}
\def \conv{{\operatorname{conv}}}

\def \cone{{\operatorname{cone}}}
\def \Vol{{\operatorname{Vol}}}

\def \MV{{\operatorname{MV}}}

\def \gr{{\operatorname{gr}}}

\DeclareMathOperator{\pyr}{pyr}

\DeclareMathOperator{\bdpyr}{pyr}

\DeclareMathOperator{\epy}{epy}

\def \latti{{\operatorname{Latt}}}
\newcommand{\Lat}[1]{\latti\bigl({#1}\bigr)}
\newcommand{\Latt}[1]{\latti\bigl({#1}\bigr)}

\def \0{\mathbf 0}
\def \fv{\mathbf v}
\def \fx{\mathbf x}

\def \fu{\mathbf u}
\def \f1{\mathbf{1}}
\def\ba{\mathbf{a}}
\def\bb{\mathbf{b}}

\def\be{\mathbf{e}}

\newcommand{\kk}{\mathbb{k}}

\def\xi{x}

\def\ls{\leqslant}
\def\gs{\geqslant}

\def\fm{\mathfrak{m}}
\def\frp{\mathfrak{p}}

\def\fp{\mathbf{p}}
\def\fq{\mathbf{q}}

\def\fq{\mathbf{q}}

\def\fc{\mathfrak{c}}

\def\bn{\mathbf{n}}

\def \QQ{\mathbb Q}
\def \RR{\mathbb R}

\def \NN{ {\mathbb Z}_{\gs 0}}

\def \ZZ{\mathbb Z}

\def \C{\mathcal C}

\def \D{\mathcal D}

\def \R{\mathcal R}

\def \M{\mathcal M}

\def \L{\mathcal L}

\def \F{\mathcal F}

\def\fx{\mathbf{x}}

\def\fe{\mathbf{e}}

\begin{document}
	
	\title[The multiplicity sequence of monomial ideals]{The multiplicity sequence of monomial ideals}

	\author[Sudipta Das]{Sudipta Das}
	\address{Sudipta Das\\
		Dr Homi Bhabha Rd, TIFR, Navy Nagar, Colaba, Mumbai, Maharashtra, 400005, India. \emph{Email:} {\rm sudiptad@math.tifr.res.in}}

	\author[Jonathan Monta{\~n}o]{Jonathan Monta{\~n}o}
	\address{Jonathan Monta{\~n}o\\School of Mathematical and Statistical Sciences, Arizona State University, P.O. Box 871804, Tempe, AZ 85287-18041, USA. \emph{Email:} {\rm montano@asu.edu}}

	\author[Aniketh Sivakumar]{Aniketh Sivakumar}
	\address{Aniketh Sivakumar\\Mathematics Department, Tulane University, 6823 St. Charles Avenue, New Orleans, LA 70118, USA. \emph{Email:} {\rm asivakumar@tulane.edu}}

	\begin{abstract}
		We give a convex-geometric formula for the multiplicity sequence of a monomial ideal in terms of mixed volumes of polytopes constructed from its Newton polyhedron. We also construct a counterexample to a conjecture of Achilles and Manaresi proposing a different volume formula for the multiplicity sequence. Finally, we derive a mixed-volume formula for the mixed multiplicities of arbitrary monomial ideals.
	\end{abstract}


	\maketitle

	\section{Introduction}
	
	The \emph{multiplicity sequence} of an ideal $I$ in a Noetherian local ring $(R,\fm)$ is the sequence of multidegrees of the bigraded algebra
	$
	\gr_{\fm}(\gr_I(R)),
	$
	obtained by taking the associated graded algebra with respect to $\fm$ of the associated graded algebra $\mathrm{gr}_I(R)$ of $R$ with respect to $I$. This sequence was introduced by Achilles and Manaresi in \cite{AchillesManaresibigraded} in connection with St\"{u}ckrad-Vogel intersection theory, and has since been studied  by both algebraic geometers and commutative algebraists; see, for example, \cite{AchillesRams01,callejas-perez-mult-seq2010,CidRPolUlr24,Ciuperca03,GaffneyGassler99,PolinietalMultSeq}.
	
	Given ideals $J\subseteq I$, a fundamental problem is to find numerical characterizations for the equality $\overline{J}=\overline{I}$ of their integral closures. When $I$ and $J$ are $\fm$-primary, Rees proved, under mild assumptions on $R$, that $\overline{J}=\overline{I}$ if and only if $I$ and $J$ have the same Hilbert-Samuel multiplicity \cite{rees1961transforms}. For ideals that are not necessarily $\fm$-primary, analogous results were later obtained using the $j$-multiplicity \cite{FlennerManaresi01} and the $\varepsilon$-multiplicity \cite{KatzValidashti10}, although the latter results require localizing $R$ at a suitable set of prime ideals. In the analytic setting, Gaffney and Gassler showed that the multiplicity sequence characterizes integral dependence \cite{GaffneyGassler99}. This result was later proved in general by Polini, Trung, Ulrich, and Validashti \cite{PolinietalMultSeq}, with one implication having previously been shown by Ciuperc\u{a} \cite{Ciuperca03}.
	
	The role of the multiplicity sequence in detecting integral dependence motivates the search for explicit formulas in concrete settings. For monomial ideals, there is a long history of results giving convex-geometric formulas for  multiplicities in terms of the Newton polyhedron of the ideal. Teissier obtained such a formula for the Hilbert-Samuel multiplicity \cite{TeissierMonomial}, while Jeffries and the second author obtained corresponding formulas for the $j$- and $\varepsilon$-multiplicities \cite{JMULT_MONOMIAL}. In a more geometric setting, Aluffi expressed the Segre class of a monomial scheme as an integral defined from  its Newton polyhedron  \cite{AluffiIntegral16}.  
	The main result of this paper provides a formula of this kind for the entire multiplicity sequence, expressed in terms of mixed volumes of polytopes constructed from the Newton polyhedron of the ideal; see \autoref{thm:main}.
	
	In \cite{AchillesManaresi2022}, Achilles and Manaresi proposed a different formula for the multiplicity sequence of a monomial ideal, also in terms of data  from its Newton polyhedron. They verified their conjecture in dimensions at most two and provided supporting examples in higher dimensions. Our next main result shows that the Achilles-Manaresi conjecture fails in dimension three; see \autoref{thm:AM-conj-false}.
	
	A closely related family of numerical invariants is given by the \emph{mixed multiplicities} associated to ideals $J,I_1,\ldots,I_r$, where $J$ is $\fm$-primary. These invariants are defined as the multidegrees of $\gr_J(\R(I_1,\ldots, I_r))$, the associated graded algebra with respect to $J$ of the multi-Rees algebra of $I_1,\ldots,I_r$, and have a wide range of applications; see \cite{TrungVermaSurvey10} for a survey on  this topic. Trung and Verma showed in \cite{TRUNG_VERMA_MIXED_VOL} that mixed volumes of polytopes can be realized as mixed multiplicities of certain monomial ideals generated in a single degree. To the best of our knowledge, however, convex-geometric formulas for the mixed multiplicities of arbitrary monomial ideals have not previously appeared in the literature. Our final main result fills this gap by providing such a formula; see \autoref{thm:mixed_mult}.

	
	\section{Algebraic preliminaries}\label{s:prelims}
	
	Although some of the definitions and results presented in this section  hold in greater generality, we restrict the exposition to the setting of our main results.

	
	\subsection{The multiplicity sequence of an ideal}\label{b:multSeq}

	Let $\kk$ be a field and let $R=\kk[x_1,\ldots, x_d]_{(x_1,\ldots, x_d)}$, with $d>0$, be a polynomial ring localized at its  homogeneous maximal ideal. Let $\fm=(x_1,\ldots, x_d)$ be the maximal ideal of $R$. 
	
	Let $I\subseteq R$ be a nonzero  ideal and  $\gr_I(R)=\bigoplus_{u\gs 0}I^u/I^{u+1}$ the   associated graded algebra of $R$ with respect to $I$. Consider the bigraded  algebra
	$$T=\gr_{\fm}(\gr_I(R))=\bigoplus_{u,v\gs 0}\frac{\fm^v I^u+I^{u+1}}{\fm^{v+1} I^u+I^{u+1}}.$$
	The second sum transform of the Hilbert function of $T$ is 
	\begin{equation*}\label{eq:H11}
		H^{(1,1)}_T(u,v)=\sum_{i=0}^u\sum_{j=0}^v\lambda_R\left(\frac{\fm^j I^i+I^{i+1}}{\fm^{j+1} I^i+I^{i+1}}\right)=\sum_{i=0}^u\lambda_R\left(\frac{I^i}{\fm^{v+1} I^i+I^{i+1}}\right).
	\end{equation*}
	It is well-known that  $H^{(1,1)}_T(u,v)$ agrees with a polynomial  $P_T^{(1,1)}(u,v)\in \QQ[u,v]$ of degree $d$   for $u,v\gg 0$ \cite{AchillesManaresibigraded}. 
	
	\begin{definition}\label{def:mult_seq}
		Writing
		$$P_T^{(1,1)}(u,v)=	\sum_{i=0}^{d}\frac{c_i(I)}{(d-i)!i!}\,u^{d-i}v^i
		+(\text{lower terms}),$$
		the coefficients $c_0(I),\cdots, c_{d}(I)$ are  nonnegative integers collectively called the \emph{multiplicity sequence of $I$}.
	\end{definition}

	\begin{remark}\label{rem:c0}
		The coefficient $c_0(I)$ coincides with the \emph{$j$-multiplicity of $I$}, originally defined in \cite{ACHILLES_MANARESI_J_MULT}. 
	\end{remark}
	
	\begin{remark}\label{rem:cd}
		The coefficient $c_d(I)$ is  zero whenever $I$ has positive codimension, which is always the case in our setting \cite[Proposition 2.3]{AchillesManaresibigraded}. Therefore, the polynomial 
		\[
		P_T^{(1,1)}(u,v)-P_T^{(1,1)}(u-1,v)=	\sum_{i=0}^{d-1}\frac{c_i(I)}{(d-i-1)!i!}\,u^{d-i-1}v^i
		+(\text{lower terms}),
		\]
		which agrees with $
		\lambda_R\left(\frac{I^u}{\fm^{v+1} I^u+I^{u+1}}\right)$ for $u,v\gg 0$, contains all the nonzero coefficients of  the multiplicity sequence.
	\end{remark}
	
	\begin{remark}\label{rem:notation_mult_seq}
		We follow the  notation from \cite{AchillesManaresibigraded, AchillesManaresi2022}.  In \cite{PolinietalMultSeq} the authors use the reverse ordering for the multiplicity sequence, i.e., their $c_i(I)$ is our $c_{d-i}(I)$.
	\end{remark}

	The following lemma shows that the  multiplicity sequence can  be computed via the filtration of integral closures of  powers. We denote by $\overline{J}$ the integral closure of an ideal $J$ of $R$. 
	
	\begin{lemma}\label{l:passing_to_int_closure}
		For $u,v\gg 0$ we have 
		$$\lambda_R\left(\frac{\overline{I^u}}{\fm^{v} \overline{I^u}+\overline{I^{u+1}}}\right)=
		\sum_{i=0}^{d-1}\frac{c_i(I)}{(d-i-1)!i!}\,u^{d-i-1}v^i+(\text{lower  terms}).
		$$
	\end{lemma}
	\begin{proof}
		By \cite[Corollary~9.2.1]{HunekeSwansonIntegral}  there exists $l\in \NN$ such that $\overline{I^{n+l}}=I^n\overline{I^l}$  for every $n\gs 0$. 
		Multiplicity sequences  can be defined with respect to a module \cite[Definition~2.4]{callejas-perez-mult-seq2010}, and the resulting invariants are additive on short exact sequences   \cite[Theorem~4.1]{callejas-perez-mult-seq2010}. Therefore, by computing multiplicity sequences along   the short exact sequence 
		$$0\to\overline{I^l}\to R\to R/\overline{I^l}\to 0,$$ using the fact that $\dim\bigl(R/\overline{I^l}\bigr)<d$, and \autoref{rem:cd}, we obtain the desired conclusion.
	\end{proof}


	\subsection{Mixed multiplicities of ideals}\label{b:mixMult} 
	
	Under the assumptions of \autoref{b:multSeq} let $J, I_1,\ldots, I_r\subseteq R$ be nonzero ideals with $J$ $\fm$-primary. Let $\R(I_1,\ldots, I_r)=\oplus_{u_1,\ldots u_r\gs 0}I_1^{u_1}\cdots I_r^{u_r}$  be the \emph{multi-Rees algebra} of $I_1,\ldots, I_r$ and consider the $\NN^{r+1}$-graded algebra
	$$
	S=\gr_J\bigl(\R(I_1,\ldots, I_r)\bigr) = \bigoplus_{u_0,\ldots u_r\gs 0} \frac{J^{u_0}I_1^{u_1}\cdots I_r^{u_r}}{J^{u_0+1}I_1^{u_1}\cdots I_r^{u_r}}.
	$$
	By \cite[\S 1]{TRUNG_VERMA_MIXED_VOL} the Hilbert polynomial of $S$ has the form
	\begin{equation}\label{eq:HilberPolS}
		P_S(u_0,\ldots, u_r)=	\sum_{n_i\gs 0,\,  \sum_in_i=d-1}\frac{e_{(n_0,\ldots, n_r)}(J|I_1,\ldots I_r)}{n_0!\cdots n_r!}\,u_0^{n_0}\cdots u_{r}^{n_r} 
		+(\text{lower terms}).
	\end{equation}
	\begin{definition}\label{def:mixed_muts}
		The nonnegative integers $e_{(n_0,\ldots, n_r)}(J|I_1,\ldots I_r)$  in \autoref{eq:HilberPolS} for $n_i\gs 0$, $\sum_in_i=d-1$ are called the \emph{mixed multiplicities} of $J,I_1,\ldots I_r$. 
	\end{definition}

	Like the multiplicity sequence, the mixed multiplicities can also be computed via the integral closures of powers, as the following lemma shows.
	
	\begin{lemma}\label{l:passing_to_int_closure_mixed_mult}
		For $u_0,\ldots, u_r\gg 0$ we have 
		$$\lambda_R\left(\frac{\overline{J^{u_0}I_1^{u_1}\cdots I_r^{u_r}}}{\overline{J^{u_0+1}I_1^{u_1}\cdots I_r^{u_r}}}\right)=
		\sum_{n_i\gs 0,\,  \sum_in_i=d-1}\frac{e_{(n_0,\ldots, n_r)}(J|I_1,\ldots I_r)}{n_0!\cdots n_r!}\,u_0^{n_0}\cdots u_{r}^{n_r} 
		+(\text{lower terms}).
		$$
	\end{lemma}
	\begin{proof}
		By \cite[Lemma 3.1]{DasRoy24} and \cite[Theorem 9.2.2]{HunekeSwansonIntegral} the algebra $\oplus_{u_0,\ldots u_r\gs 0}\overline{J^{u_0}I_1^{u_1}\cdots I_r^{u_r}}$ is a finitely generated  $\R(J,I_1,\ldots, I_r)$-module. Thus, there exists $l\gg 0$ such that  for every $u_0,\ldots, u_r\gs l$ we have
		\begin{equation}\label{eq:powers_mixed_mult}
			\overline{J^{u_0}I_1^{u_1}\cdots I_r^{u_r}} = J^{u_0-l}I_1^{u_1-l}\cdots I_r^{u_r-l}\overline{J^lI_1^{l}\cdots I_r^{l}}.
		\end{equation}
		Mixed multiplicities  can be defined with respect to a module \cite[Note 2.1]{VietThanh14}, and the resulting invariants are additive on short exact sequences   \cite[Corollary 3.9]{VietThanh14}. Therefore,
		by computing mixed multiplicities  along   the short exact sequence 
		$$0\to\overline{J^lI_1^{l}\cdots I_r^{l}}\to R\to R/\overline{J^lI_1^{l}\cdots I_r^{l}}\to 0,$$
		using \autoref{eq:powers_mixed_mult}, and the fact that $\dim\bigl(R/\overline{J^lI_1^{l}\cdots I_r^{l}}\bigr)<d$, we obtain the desired conclusion.
	\end{proof}


	\section{Convex-geometric preliminaries}\label{s:conv_Geom} 
	
	For sets $A,B\subseteq \RR^d$ we denote by $A+B=\{a+b\mid a\in A,\, b\in B\}$ their Minkowski sum. We denote by $\Vol_d(A)$ the $d$-dimensional volume of $A$. 
	For simplicity, we use the following notation throughout the paper.
	
	\begin{notation}\label{not:lattice_counting}
		For a subset $S\subseteq \RR^d$, we write $\Lat{S}$ for the set of lattice points in $S$, i.e.,
		\[
		\Lat{S}= S\cap \ZZ^d.
		\]
	\end{notation}
	
	A function $f:\ZZ^d\to \QQ$ is a \emph{quasi-polynomial} if there exists a subgroup $\Lambda\subseteq \ZZ^d$ of finite index and one polynomial $P_{\overline{\bn}}$ for each class $\overline{\bn}=\bn+\Lambda \in \ZZ^d/\Lambda$ such that $f(\bn)=P_{\overline{\bn}}(\bn)$ for every $\bn\in \ZZ^d$.

	\begin{thmdfn}\label{thm:mixed_Ehrhart_structure}
		Let $\mathbf{Q}=Q_1,\ldots, Q_r\subseteq \RR^d$ be rational polytopes. The function  
		$$
		E_{\mathbf{Q}}(t_1,\ldots,t_r)
		:=
		\#\Lat{t_1Q_1+\cdots+t_rQ_r}
		$$
		coincides with a quasi-polynomial of degree $\dim(Q_1+\cdots+Q_r)$ for every $t_1,\ldots,t_r\in \NN$.
		The homogeneous degree $d$ term of $E_{\mathbf{Q}}(t_1,\ldots,t_r)$ has constant coefficients and it is equal to 
		$$\Vol_d(t_1Q_1+\cdots+t_rQ_r) = 
		\sum_{n_i\gs 0,\,  \sum_in_i=d}
		\binom{d}{n_1,\ldots, n_r}\MV_d\bigl(Q_1[n_1],\ldots, Q_r[n_r]\bigr)t_1^{n_1}\cdots t_r^{n_r},
		$$
		where $\MV_d\bigl(Q_1[n_1],\ldots, Q_r[n_r]\bigr)$ is the \emph{mixed volume} of the polytopes $Q_1,\ldots, Q_1,\ldots, Q_r,\ldots,Q_r$ with each $Q_i$ repeated $n_i$ times 
		
		Furthermore,  if $Q_1,\ldots, Q_r$  have lattice vertices, then $E_{\mathbf{Q}}(t_1,\ldots,t_r)$ has constant coefficients. 
	\end{thmdfn}
	\begin{proof}
		From \cite[Theorem 7 and \S4]{McMullenLattice78} it follows that $E_{\mathbf{Q}}(t_1,\ldots,t_r)$ is a quasi-polynomial of degree  $\dim(Q_1+\cdots+Q_r)$ whose coefficients in degree $d$ are constant. It also follows from this reference that $E_{\mathbf{Q}}(t_1,\ldots,t_r)$ has constant coefficients when $Q_1,\ldots, Q_r$ have lattice vertices.    
		By Riemann integration we have $\Vol_d(t_1Q_1+\cdots+t_rQ_r)=
		\lim_{t\to\infty} t^{-d}E_{\mathbf{Q}}(tt_1,\ldots,tt_r)$. The rest of the statement now follows from Minkowski Theorem; see for example 
		\cite[Theorem IV.3.2]{EwaldCombConvAlgGeom}.
	\end{proof}
	
	\begin{remark}\label{rem:mixed_vol}
		In the mixed volume definition in \autoref{thm:mixed_Ehrhart_structure} we have  $\MV_d\bigl(Q^{[d]}\bigr)=\Vol_d(Q)$. The reader is warned that some
		authors use the normalized  version, under which  mixed volumes are
		$d!$ times ours.
	\end{remark}
	
	\begin{remark}\label{rem:empty_zero_MV}
		If $B=\{\0\}$, then  $A+B=A$.   Therefore,  $\MV_d\bigl(Q_1[n_1],\ldots, Q_{r-1}[n_{r-1}], \{\0\}[n_r]\bigr)$ is equal to $\MV_d\bigl(Q_1[n_1],\ldots, Q_{r-1}[n_{r-1}]\bigr)$ when $n_r=0$, and equal to zero otherwise. 
		On the other hand, if $B=\emptyset$, then  $A+B=\emptyset$.    Therefore, we may assume by convention that $\MV_d\bigl(Q_1[n_1],\ldots, Q_{r-1}[n_{r-1}],  \emptyset[n_r]\bigr)=0$ for every $n_1,\ldots, n_r$.
	\end{remark}
	
	A cone   $C\subseteq \RR^d$ with apex at the origin $\0$  is \emph{pointed} if $C\cap (-C)=\{\0\}$. The cone $C$ is \emph{polyhedral} if $C=\RR_{\gs 0}\,\conv(\fu_1,\ldots, \fu_r)$ for some  $\fu_i\in \RR^d$. 
	
	The following lemma plays a fundamental role in the proof  of the main results.

	\begin{lemma}\label{l:convex_stuff}
		Let $C\subset \RR^d$ be a polyhedral pointed cone. Let $\Lambda_1=P_1+C$ and $\Lambda_2=P_2+C$ be two polyhedra in $\RR^d$, where $P_1$ and $P_2$ are polytopes. Further assume $\Lambda_2\subsetneq C$ and that $C\setminus \Lambda_2$ is bounded. Then, $C=\RR_{\gs 0}P_2$ and the following  equality holds 
		$$\Lambda_1\setminus (\Lambda_1+\Lambda_2)
		=
		\bigl(P_1+\conv(\0,P_2)\bigr)\setminus(P_1+P_2).
		$$
	\end{lemma}
	\begin{proof} 
		Let $\fu$ be a point on an extremal ray of $C$, that is, on a one-dimensional face of $C$. By the boundedness assumption, there exists $r\in \RR_{>0}$ with $r\fu\in \Lambda_2$, and so  $r\fu = \fp + \fc$ for some $\fp\in P_2$, $\fc\in C$. Since $\fu$ is extremal, we must have that $\fp$ belong to the same ray as $\fu$. We conclude that $C=\RR_{\gs 0}P_2$, proving the first part of the statement. 
		
		For the second part, we first note that from $C=\RR_{\gs 0}P_2$ we obtain 
		$C=\conv(\0,P_2)\cup \Lambda_2$. Therefore, 
		since  $\Lambda_1+\Lambda_2=P_1+\Lambda_2$, we have $$\Lambda_1=(P_1+\conv(\0,P_2))\cup (P_1+
		\Lambda_2)=(P_1+\conv(\0,P_2))\cup (\Lambda_1+
		\Lambda_2).$$ We now show that $(P_1+\conv(\0,P_2))\cap (P_1+
		\Lambda_2)=P_1+P_2$, which finishes the proof as $\Lambda_1+\Lambda_2=P_1+\Lambda_2$. The containment $\supseteq$ is clear.  Let $\fp_1,\fq_1\in P_1$, $\fp_2\in \conv(\0,P_2)$, $\fq_2\in \Lambda_2$ be such that $\fp_1+\fp_2=\fq_1+\fq_2$. It follows that 
		\begin{equation}\label{eq:p1p2}
			(\lambda\fp_1+(1-\lambda)\fq_1)+(\lambda\fp_2+(1-\lambda)\fq_2)=\fp_1+\fp_2
			\qquad \text{for all } 
			\lambda\in[0,1].
		\end{equation}
		Since $C=\RR_{\gs 0}P_2$, we have $\Lambda_2=\RR_{\gs 1}P_2$.
		Hence the segment joining $\fp_2$ and $\fq_2$ intersects  $P_2$, i.e., there
		exists $\lambda\in [0,1]$ with $\lambda\fp_2+(1-\lambda)\fq_2\in P_2$. As $P_1$
		is convex, $\lambda\fp_1+(1-\lambda)\fq_1\in P_1$, and it follows from \autoref{eq:p1p2} that
		$\fp_1+\fp_2\in P_1+P_2$.	
	\end{proof}

	\begin{proposition}\label{p:passing_to_Ehrhart}
		Under the   hypotheses of \autoref{l:convex_stuff}, assume further that  $P_1$ and $P_2$ have rational vertices.  For any $u,v \in \ZZ_{>0}$ we have $$
		\#\Lat{u\Lambda_1\setminus(u\Lambda_1+v\Lambda_2)}
		=
		E_{P_1,\conv(\0,P_2)}(u,v)-E_{P_1,P_2}(u,v).
		$$
	\end{proposition}
	\begin{proof}
		Since $uP_1+vP_2\subseteq uP_1+\conv(\0,vP_2)=uP_1+v\,\conv(\0,P_2)$  for every $u,v \in \ZZ_{>0}$, 
		the result follows directly from \autoref{l:convex_stuff} and \autoref{thm:mixed_Ehrhart_structure}.
	\end{proof}
	
	From \autoref{p:passing_to_Ehrhart} and \autoref{thm:mixed_Ehrhart_structure} we obtain the following corollary.
	
	\begin{corollary}\label{cor:asymp_dif_Ehr}
		Under the same  assumptions in \autoref{p:passing_to_Ehrhart}, if  $P_1$ and $P_2$ have lattice vertices, then 
			$
			\#\Lat{u\Lambda_1\setminus(u\Lambda_1+v\Lambda_2)}		
			$
			agrees with a polynomial of degree   $\dim(P_1+\conv(\0,P_2))$ for every $u,v\in \ZZ_{>0}$.
		\end{corollary}
		
		The next definition provides a way to capture the asymptotic growth of 
		functions in two variables.
		\begin{definition}\label{d:o(ell)}
			Let $N$ be a non-negative integer and let $f:\ZZ_{>0}^2\to \RR$ be a function.
			We say that $f= o(N)$ if 
			\[
			\lim_{t\to\infty}
			\frac{f(tu+a,tv+b)}{t^{N}}=0\,\,\text{ for every }\,\,u,v\in \ZZ_{>0},\,\,a,b,\in \ZZ.
			\]
			We write $f(u,v)=f'(u,v)+o(N)$ if   $f-f'\in o(N)$.  
		\end{definition}
		
		The following simple remark about \autoref{d:o(ell)} is used in several of the proofs below. 
		


		\begin{remark}\label{rem:complexitydown}
			If $f=p+o(N)$ for some polynomial $p\in \RR[u,v]$ of degree at most $N$, then   $f(u+c,v)=f(u,v)+o(N)$ for any $c\in\ZZ$.
		\end{remark}
		


		The set of lattice points $\Lat{u\Lambda_1\setminus(u\Lambda_1+v\Lambda_2)}$  is contained in $\Lat{u\Lambda_1}\setminus(\Latt{u\Lambda_1}+v\Lambda_2)$, and the inclusion is generally strict, since a lattice point of a Minkowski sum need not be a sum of lattice points. The next proposition shows that the discrepancy is negligible for our purposes.  
		We denote by 
		$\fe_1,\ldots, \fe_d$ 
		the standard basis of
		$\RR^d$.

		\begin{proposition}\label{p:Latt_inside}
			Under the same  assumptions in \autoref{p:passing_to_Ehrhart}, further assume that $P_1$ has lattice vertices and that $P_2 = \conv(\fe_{i_1}, \ldots, \fe_{i_r})$ for some $1\ls i_1,\ldots, i_r\ls d$. We have
			\[
				\#\Lat{u\Lambda_1\setminus(u\Lambda_1+v\Lambda_2)}
				=
				\#\bigl(\Lat{u\Lambda_1}\setminus(\Latt{u\Lambda_1}+v\Lambda_2)\bigr)+o\bigl(\dim(P_1+\conv(\0,P_2))\bigr).
			\]
		\end{proposition}
		\begin{proof}
			Without loss of generality we assume $i_j=j$ for every $j$. By \autoref{l:convex_stuff}, one has $C=\RR_{\gs0}P_2=\RR_{\gs 0}\conv(\fe_{1}, \ldots, \fe_{r})$, and so $\Lambda_2 = \RR_{\gs 1}\conv(\fe_{1}, \ldots, \fe_{r})$.  
			
			To simplify notation, set
			\[
			g(u,v)=\#\Lat{u\Lambda_1\setminus(u\Lambda_1+v\Lambda_2)}
			\qquad\text{and}\qquad
			h(u,v)=\#\bigl(\Lat{u\Lambda_1}\setminus(\Lat{u\Lambda_1}+v\Lambda_2)\bigr).
			\]
			Since  $g(u,v)\ls h(u,v)$, by 
			\autoref{cor:asymp_dif_Ehr} and \autoref{rem:complexitydown}, it suffices to show 
			$h(u,v)\ls g(u,v+c)$ for some $c\in \NN$ and every $u,v\in \ZZ_{>0}$. We show below that  
			$
			\Lat{u\Lambda_1+(v+r)\Lambda_2}
			\subseteq 
			\Lat{u\Lambda_1}+v\Lambda_2
			$ for every $u,v\in \ZZ_{>0}$, which finishes the proof.

			Let $\ba=(a_1,\ldots, a_d)\in u \Lambda_1$ and $\bb\in (v+r)\Lambda_2$ be such that $\ba+\bb\in \ZZ^d$.   Let $\lceil\ba\rceil = (\lceil a_1\rceil, \ldots, \lceil a_d\rceil)$, where $\lceil a_i\rceil$ is the ceiling of $a_i$. Observe that $a_i\in \ZZ$ for $i> r$, so that $\lceil\ba\rceil -\ba\in C$ and $\lceil\ba\rceil\in u\Lambda_1$. Let $\bb'\in \NN^d$ be such that $\ba+\bb=\lceil\ba\rceil +\bb'$. Thus,  $|\bb'|=|\bb|-|\lceil\ba\rceil-\ba|\gs (v+r)-r=v$. We conclude  $\bb'\in v\Lambda_2$. It follows that $\ba+\bb=\lceil\ba\rceil +\bb'\in  \Latt{u\Lambda_1}+v\Lambda_2$, as desired. 
		\end{proof}

		
		\section{Formula for the multiplicity sequence}\label{s:mult-seq}

		Let $\kk$ be a field and $d\in \ZZ_{>0}$. 
		Let $I=(\fx^{v_1},\ldots, \fx^{v_r})\subsetneq \kk[x_1,\ldots,x_d]_{(x_1,\ldots, x_d)}$ be a nonzero  monomial ideal generated by the monomials $\fx^{v_i}=x_1^{v_{i,1}}\cdots x_d^{v_{i,d}}$. 
		Let $$\Gamma=\conv(\fv \mid \fx^\fv\in I)\subseteq \RR^d$$ be the \emph{Newton polyhedron of $I$}. 
		We 
		say a face of $\Gamma$ is \emph{non-coordinate} 
		if it is not contained in any of the coordinate planes. 
		For each $n\in \NN$ we denote by $\F_{n}(\Gamma)$ the set of non-coordinate faces of $\Gamma$ of dimension $n$.  We note that $\F_{d-1}(\Gamma)$ is the set of  \emph{non-coordinate facets} of $\Gamma$. We denote by  
		$\F(\Gamma)=
		\bigcup_{n < d}\F_n(\Gamma)$  the set of all non-coordinate proper faces of $\Gamma$.

		Fix $F\in \F(\Gamma)$ and let $\bb_1,\ldots, \bb_g$ be the vertices of $F$. The face  $F$  can be written as
		\begin{equation}\label{eq:face_decomp}
			F=\conv(\bb_1,\ldots, \bb_g)+\RR_{\gs 0}\fe_{u_1}+\cdots +\RR_{\gs 0}\fe_{u_k}.
		\end{equation}  
		Note that $k=0$ if and only if $F$ is bounded. 
		We call the vectors $\fe_{u_i}$ the \emph{unbounded directions of $F$}, and the other  $\fe_i$ the \emph{bounded directions of $F$}. 
		
		We  define 
		$$
		\Delta_F=\conv(\fe_{u_1},\ldots, \fe_{u_k}),
		\qquad 
		\Delta^{\leqslant}_F = \conv(\0, \Delta_F),
		\qquad  \text{and}   \qquad 
		\Delta_F^{\geqslant}=\RR_{\gs 1}\Delta_F.
		$$
		If $F$ is bounded, we set by convention 
		\begin{equation}\label{eq:conv_bounded}
			\Delta_F^{\ls}=\{\0\} 
			\qquad
			\text{and}
			\qquad 
			\Delta_F=\Delta_F^{\geqslant}=\emptyset. 
		\end{equation}

		It is also convenient to define the \emph{polytope of $F$} to be the convex hull of its vertices 
		$$\pol(F)= \conv(\bb_1,\ldots, \bb_g ),$$ 
		and two natural pyramids: 
		the \emph{pyramid} over $\pol(F)$, 
		and the \emph{extended pyramid} over the entire $F$
		$$
		\bdpyr(F)=\conv(\0,\pol(F))
		\quad \text{and} \quad 
		\epy(F)= \overline{\conv(\0,F)}=
		\bdpyr(F)+\RR_{\gs 0}\Delta_F.
		$$
		For each $u\in \ZZ_{>0}$ we set 
		\begin{equation}\label{eq:coneu}
		\cone_{u}(F)  = 
		\bigcup_{u\ls s< u+1}s\,F
		=
		\left((u+1)\epy(F) \setminus (u+1)F\right) \setminus \left(u\epy(F) \setminus uF\right).
		\end{equation}
		In other words, $\cone_u(F)$ consists of the portion of the cone over  $\epy(F)$  
		lying between the $u$-th and $(u+1)$-st dilates of $F$,  and not including $(u+1)F$. 

		We illustrate the notation introduced so far with a simple example.
		\begin{example}\label{e:notation_cones}
			Let $I=(x_1^2x_3,x_2^2x_3,x_1x_2)\subset \kk[x_1,x_2,x_3]_{(x_1,x_2,x_3)}$. Set $\fv_1=(2,0,1)$, $\fv_2=(0,2,1)$, and $\fv_3=(1,1,0)$.  The Newton polyhedron of $I$ is 
			$$\Gamma=\conv(\fv_1,\fv_2,\fv_3)+\RR_{\gs 0}^3.$$
			In \autoref{fig:NP} we plot $\Gamma$ and its non-coordinate facets. 
			Writing these facets 
			as in \autoref{eq:face_decomp} we obtain
			\[
			F_1 = \conv(\fv_1, \fv_3) + \RR_{\gs 0}\fe_1, \qquad 
			F_2 = \conv(\fv_1, \fv_2, \fv_3) + \RR_{\gs 0}\fe_3, \qquad
			F_3 = \conv(\fv_2, \fv_3) + \RR_{\gs 0}\fe_2.
			\]
			The only unbounded direction of $F_1$ is $\fe_1$, therefore
			$$\Delta_{F_1} = \{\fe_1\},\qquad \Delta_{F_1}^{\leqslant}=\conv(\0,\fe_1),\qquad \text{and}  \qquad \Delta_{F_1}^{\gs }=\RR_{\gs 1}\fe_1.$$ 
			The polytope  $\pol(F_1)$ is the segment joining  $\fv_1$ and  $\fv_3$. The pyramid $\bdpyr(F_1)$ is  the triangle with vertices $\0, \fv_1,\fv_3$, and the extended pyramid $\epy(F_1)$ is 
			$\bdpyr(F_1)+\RR_{\gs 0}\fe_1$.  
			For $u\in \RR_{\gs 0}$, we have $\cone_{u}(F_1)=Q_u+\RR_{\gs 0}\fe_1$, where $Q_u$ is the quadrilateral with vertices $u\fv_1$, $u\fv_3$, $(u+1)\fv_1$, $(u+1)\fv_3$  with the segment joining $(u+1)\fv_1$ and $(u+1)\fv_3$  removed. 
			\begin{figure}[H]
				\centering
				\includegraphics*[height=2.5in]{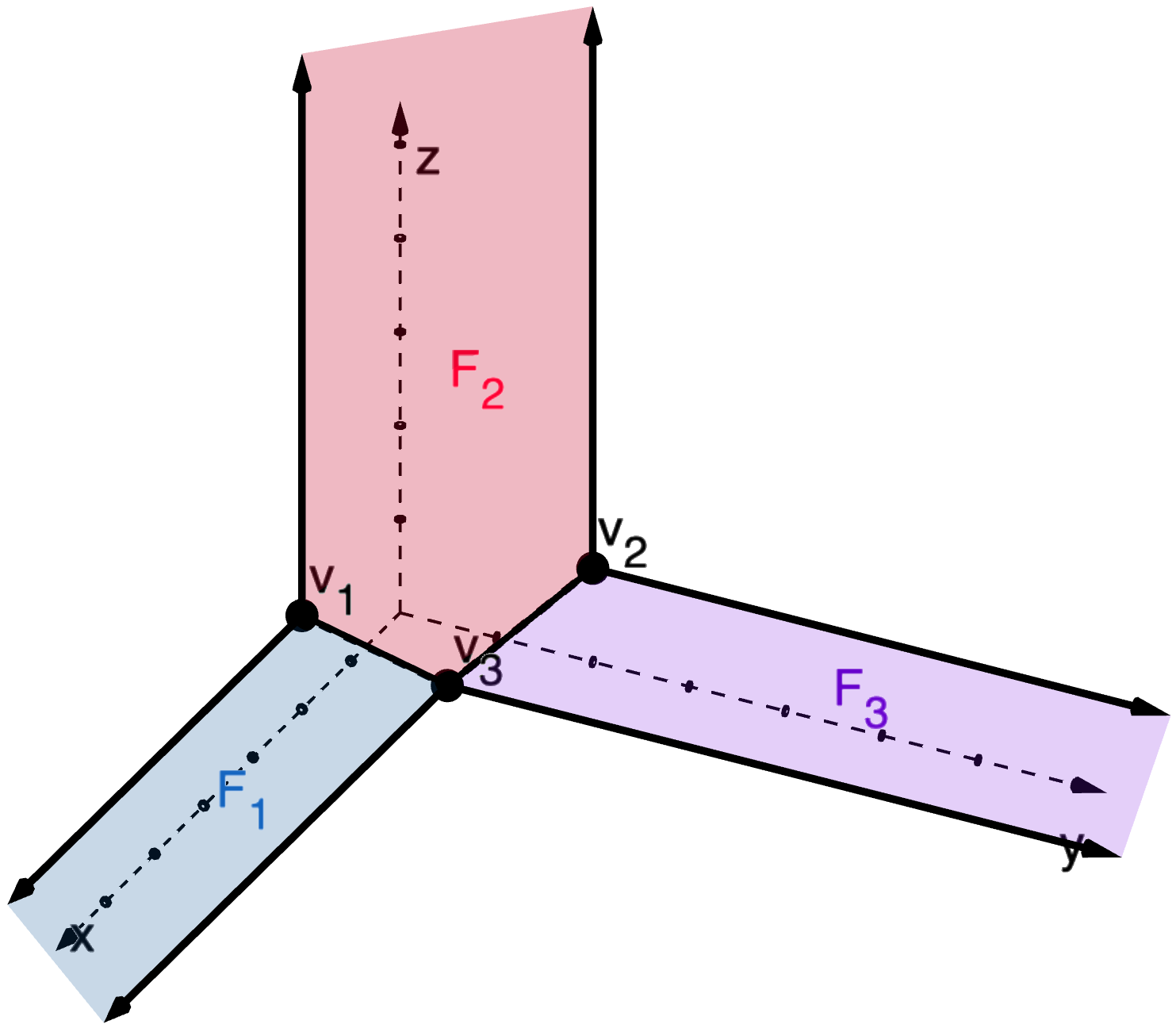}
				\caption{The Newton polyhedron of $I=(x_1^2x_3,x_2^2x_3,x_1x_2)$ and its  non-coordinate  facets.}
				\label{fig:NP}
			\end{figure}
			
		\end{example}

		 For a submonoid $S\subseteq \NN^d$ we denote by $\kk[S]$ its \emph{monoid algebra}, that is, the subalgebra of $\kk[x_1,\ldots, x_d]$ generated by the of monomials $\{\fx^{\fv}\mid \fv \in S\}$.
		
		\begin{lemma}\label{l:polyuvgg0}
				Let $F\in \F(\Gamma)$ be an unbounded face. The function 
			\[		\#\left[\Lat{u\epy(F)}
			\setminus 
			\left(
			\Lat{u\epy(F)} + 
			v\Delta_F^{\gs}
			\right)\right]
			\] 
			agrees with a polynomial for $u,v\gg 0$.
		\end{lemma}
		\begin{proof} 
		Assume without loss of generality that the unbounded directions of $F$ are $\fe_1, \ldots, \fe_k$ and 
		  let $\bb_1,\ldots, \bb_g$ be the vertices of $\pol(F)$.  
		  
		  Let $S$ be the submonoid of $ \NN^{d+1}$ generated by the vectors 
		  \[
		  \{(\fe_1,0), \ldots, (\fe_k,0),(\0,1),(\bb_1,1),\ldots, (\bb_g,1)\}.
		  \] 
		  Consider the monoid  $\tilde{S}=\RR_{\gs 0}S\cap \ZZ^{d+1}$, and  let $A:=\kk[S]\subseteq B:=\kk[\tilde{S}]\subseteq \kk[x_1,\ldots, x_d,t]$ be the corresponding monoid algebras, where the rightmost coordinates correspond to  the variable $t$. We use  $t$ to give $\ZZ$-gradings to $A$ and $B$, that is,  $\deg(\fx^\fv t^u) =u$, and write $A=\oplus_{u\in \NN}A_u$,  $B=\oplus_{u\in \NN}B_u$.
		 
		  Let $\mathfrak{p}\subset A$ be the ideal generated by $(x_1,\ldots, x_k)$. Let   $\gr_\frp(A)=\oplus_{v\in \NN}\frp^v/\frp^{v+1}$ and  $\gr_\frp(B)=\oplus_{v\in \NN}\frp^vB/\frp^{v+1}B$ be the associated graded algebras of $A$ and $B$ with respect to $\frp$. 
		  Observe that $\gr_\frp(A)=\oplus_{u,v\in \NN}\frp^vA_u/\frp^{v+1}A_u$ has the structure of a standard bigraded algebra  over $A_0/\frp A_0=\kk$, and $\gr_\frp(B)=\oplus_{u,v\in \NN}\frp^vB_u/\frp^{v+1}B_u$ of  a bigraded $\gr_\frp(A)$-module. 

		   Since $B$ is a finitely generated $A$-module \cite[Corollary 2.10 and Proposition 2.7(b)]{BrunsGubeladze09}, 
		 $\gr_\frp(B)$ is a finitely generated $\gr_\frp(A)$-module. Therefore, the first sum transform of its  Hilbert function 
		  \[
			\dim_\kk\left(B_u/\frp^{v}B_u\right) = \sum_{i=0}^{v-1}\dim_\kk\left(\frp^{i}B_u/\frp^{i+1}B_u\right) 
		  \]
		  agrees with a polynomial for $u,v\gg 0$. 
		 
		 Finally, the result follows by observing that 
		 \[
		 \tilde{S}=\{(\bb,0)\mid \bb\in \Delta_F^{\gs }\}
		 \cup \{(\bb,u)\mid u\in \ZZ_{>0}, \bb\in u\epy(F)\},
		 \] 
		 so that 
		 \[
		 \dim_\kk\left(B_u/\frp^{v}B_u\right) = 	\#\left[\Lat{u\epy(F)}
		 \setminus 
		 \left(
		 \Lat{u\epy(F)} + 
		 v\Delta_F^{\gs}
		 \right)\right]
		 \]
		 for every $u,v\in \ZZ_{>0}$. 
		\end{proof}

		For $F\in \F(\Gamma)$ and $u,v\in \ZZ_{>0}$ we  define the function
		\begin{equation}\label{eq:Lfunction}
			\L_F(u,v)=
			\#\left[
			\Lat{\cone_u(F)}
			\setminus 
			\left(
			\Lat{\cone_u(F)} + v\Delta_F^{\gs }
			\right)
			\right].
		\end{equation}
		
		\begin{proposition}\label{p:EG_pol}
			Let $F\in \F(\Gamma)$. Then
			\begin{align*}
				\L_F(u,v)&=
				\bigl(
				E_{\bdpyr(F), \Delta_F^{\ls}}\bigl(u+1,v\bigr) &&\hspace{-10pt}-
				E_{\bdpyr(F), \Delta_F^{\ls}}\bigl(u,v\bigr)
				\bigr)\\ 
				&&&-\bigl(
				E_{\bdpyr(F), \Delta_F}\bigl(u+1,v\bigr) - 
				E_{\bdpyr(F), \Delta_F}\bigl(u,v\bigr)
				\bigr)
				+ o(\dim(F))\\
				&=o(\dim(F)+1)&&
			\end{align*}
		\end{proposition}
		\begin{proof}
			If $F$ is bounded, our convention in  \autoref{eq:conv_bounded} implies 
			$\L_F(u,v)=
			\#
			\Lat{\cone_u(F)}$. Therefore,
			\[
			\L_F(u,v)
			=
			\bigl(\#\Lat{(u+1)\pyr(F)}-\#\Lat{u\pyr(F)}\bigr) - 
			\bigl(\#\Lat{(u+1)F}-\#\Lat{uF}\bigr),
			\] 
			and then the conclusion follows from \autoref{thm:mixed_Ehrhart_structure},  \autoref{rem:empty_zero_MV}, and \autoref{rem:complexitydown}.
			
			For the rest of the proof we assume $F$ is unbounded. 
			Since $sF+v\Delta_F^{\gs}\subseteq sF$ for every $s\in \RR_{>0}$ and $v\in \ZZ_{>0}$, from \autoref{eq:coneu} we obtain
			\begin{align}
			\label{eq:LL}\L_F(u,v)=
			&\Bigl(\#\left[\Lat{(u+1)\epy(F)}
			\setminus 
			\left(
			\Lat{(u+1)\epy(F)} + 
			v\Delta_F^{\gs}
			\right)\right]\\ 	\nonumber
			&\phantom{MMMMMMMMMMM}
			-\#\left[\Lat{u\epy(F)}
			\setminus 
			\left(
			\Lat{u\epy(F)} + 
			v\Delta_F^{\gs}
			\right)\right]\bigr)
			\\ \nonumber
			&\phantom{MM}
			-\Bigl(\#\left[\Lat{(u+1)F}
			\setminus 
			\left(
			\Lat{(u+1)F} + 
			v\Delta_F^{\gs}
			\right)\right]\\ 	 \nonumber
			&\phantom{MMMMMMMMMMMMMMMM}
			-\#\left[\Lat{uF}
			\setminus 
			\left(
			\Lat{uF} + 
			v\Delta_F^{\gs}
			\right)\right]\bigr)
			\end{align}
		Let us denote by $\Sigma_1$ and $\Sigma_2$ the first and second differences on the right hand side of \autoref{eq:LL}, so that $\L_F(u,v)=\Sigma_1-\Sigma_2$. 
			
		Applying \autoref{p:Latt_inside}  with $P_1 = \pol(F)$ and $P_2 = \Delta_F$ we obtain 
			\[
			\#\left[\Lat{uF}
			\setminus 
			\left(
			\Lat{uF} + 
			v\Delta_F^{\gs}
			\right)\right] = 
			\#\Lat{
				uF\setminus (uF+ v\Delta_F^{\gs})
			}
			+ o(\dim(F)).
			\]
		Thus, by \autoref{cor:asymp_dif_Ehr} and \autoref{rem:complexitydown}, $\Sigma_2 = o(\dim(F))$.

		Applying \autoref{p:Latt_inside} with $P_1 = \bdpyr(F)$ and $P_2 = \Delta_F$ we obtain
			\[
			\#\left[\Lat{u\epy(F)}
			\setminus 
			\left(
			\Lat{u\epy(F)} + 
			v\Delta_F^{\gs}
			\right)\right]
			=
			\#\Lat{
				u\epy(F)\setminus (u\epy(F)+ v\Delta_F^{\gs})
			}
			+ o(\dim(F)+1);
			\]
			indeed, 
			\begin{equation}\label{eq:dim_compu}
				\dim\bigl(\bdpyr(F)+\Delta_F^{\ls}\bigr)\ls \dim\bigl(\pol(F)+\Delta_F^{\ls}\bigr)+1=\dim(F)+1.
			\end{equation}
			Furthermore, applying \autoref{p:passing_to_Ehrhart} 
			with $P_1 = \bdpyr(F)$ and $P_2 = \Delta_F$ we obtain 
			\[
			\#\Lat{
				u\epy(F)\setminus (u\epy(F)+ v\Delta_F^{\gs})} = 
			E_{\bdpyr(F), \Delta_F^{\ls}}\bigl(u,v\bigr) -
			E_{\bdpyr(F), \Delta_F}\bigl(u,v\bigr).
			\]
			Thus, by \autoref{l:polyuvgg0}, \autoref{rem:complexitydown}, \autoref{thm:mixed_Ehrhart_structure}, and \autoref{eq:dim_compu}, 
				\begin{align*}
				\Sigma_1&=
				\bigl(
				E_{\bdpyr(F), \Delta_F^{\ls}}\bigl(u+1,v\bigr) &&\hspace{-10pt}-
				E_{\bdpyr(F), \Delta_F^{\ls}}\bigl(u,v\bigr)
				\bigr)\\ 
				&&&-\bigl(
				E_{\bdpyr(F), \Delta_F}\bigl(u+1,v\bigr) - 
				E_{\bdpyr(F), \Delta_F}\bigl(u,v\bigr)
				\bigr)
				+ o(\dim(F))\\
				&=o(\dim(F)+1),&&
			\end{align*}
		finishing the proof.
		\end{proof}

		For the proof of our next proposition we need the following  lemmas with general facts about the dilations of $\Gamma$. We first introduce some notation.
		
		\begin{notation}\label{not:supp_hyp}
			For each $F\in \F_{d-1}(\Gamma)$ we denote by $\ell_F^{-1}(1)$ the \emph{normalized supporting hyperplane of $F$}, that is, $\ell_F:\RR^d\to \RR$  is a linear form and $F=\Gamma \cap \ell_F^{-1}(1)$. In particular, $$\Gamma =
			\bigl\{\ba\in \RR^d_{\gs0}\mid \ell_F(\ba)\gs 1 \text{ for every }F\in \F_{d-1}(\Gamma)
			\bigr\}.$$
		\end{notation}
		
		\begin{remark}\label{r:bdunbd}
			$\fe_i$ is an unbounded direction of $F\in \F_{d-1}(\Gamma)$ if and only if $\ell_F(\fe_i)=0$. Since $\ell_F$ is normalized, we have $\ell_F(\fe_i)>0$ if $\fe_i$ is bounded. 
		\end{remark}
		
		\begin{lemma}\label{l:uGamma}
			Let  $u\in \ZZ_{>0}$.
			\begin{enumerate}
				\item[{\rm (1)}] Let $F\in \F_{d-1}(\Gamma)$. Then,  $\ba\in \RR_{> 0}F$ if and only if  $\ba\in \RR^d_{\gs0}$ and $0<\ell_F(\ba)\ls \ell_{G}(\ba)$ for every $G\in \F_{d-1}(\Gamma)$. 
				\item[{\rm (2)}] $
				u\Gamma\setminus (u+1)\Gamma
				=
				\bigcup_{F\in \F_{d-1}(\Gamma)} \cone_u(F).
				$
			\end{enumerate} 
		\end{lemma}
		\begin{proof}
			(1): Since either side of the equivalence implies that $s:=\ell_F(\ba)>0$, we may assume this throughout. Now, $\ba\in \RR_{>0}F$, if and only if $s^{-1}\ba\in F$,  if and only if $\ell_{F}(s^{-1}\ba)=1\ls \ell_{G}(s^{-1}\ba)$ for every $G\in \F_{d-1}(\Gamma)$, if and only if $\ell_{F}(\ba)=s\ls \ell_{G}(\ba)$ for every 
			$G\in \F_{d-1}(\Gamma)$.
			
			(2): Let $\ba\in \cone_u(F)\subset u\Gamma$ for some $F\in \F_{d-1}(\Gamma)$. Then $u\ls \ell_F(\ba)<u+1$ and, by (1),   $\ell_{G}(\ba)\gs \ell_F(\ba)\gs u$ for any other $G\in \F_{d-1}(\Gamma)$. Thus, $\ba\in u\Gamma\setminus (u+1)\Gamma$.
			
			Conversely, let $\ba\in u\Gamma\setminus (u+1)\Gamma$, then $\ell_G(\ba)\gs u$ for every $G\in \F_{d-1}(\Gamma)$. 
			Let $F$ be the facet where $\min\{\ell_G(\ba)\mid G\in\F_{d-1}(\Gamma)\}$ is achieved.
			Since $\ba\in u\Gamma\subseteq \RR_{\gs 0}^d$ and $\ba \not\in (u+1)\Gamma$, we  have $\ell_{G}(\ba)<u+1$ for some $G\in \F_{d-1}(\Gamma)$. By minimality, $\ell_F(\ba)<u+1$. Thus, from (1) we obtain  $\ba \in \cone_u(F)$.
		\end{proof}
		
		\begin{lemma}\label{l:bounded_dir_goes_out}
			Let $F\in \F_{d-1}(\Gamma)$.
			\begin{enumerate}
				\item[{\rm (1)}]  Let   $\fe_i$ be a bounded direction of $F$. There exists $t_{F,i}\in \ZZ_{>0}$ such that 
				$\ba+v\fe_i\not\in\cone_u(F)$ for every   $u\in \ZZ_{>0}$, $\ba\in u\Gamma$, 
				and integer $v\gs t_{F,i}$. 
				
				\item[{\rm (2)}] Let $\fe_j$ be an unbounded direction of $F$. There exists $l_{F,j}\in \NN$ 
				such that if $\ba+v\fe_j\in \cone_u(F)$   for some $u\in  \ZZ_{>0}$, $v\in \NN$, and $\ba\in u\Gamma\setminus (u+1)\Gamma$, then $\ba+l_{F,j}\,\fe_j\in \cone_u(F)$. 
			\end{enumerate}
		\end{lemma}
		\begin{proof}
			
			We first prove (1). 
			Since $\fe_i$ is a bounded direction of $F$, one has $\ell_F(\fe_i)>0$. If  $\ba\in u\Gamma$ and $v\gs t_{F,i}:=\lceil\ell_F(\fe_i)^{-1}\rceil$, then $\ell_F(\ba+v\fe_i)
			\gs u+1
			$, 
			and so $\ba+v\fe_i\not\in \cone_u(F)$.

			We continue with (2). Fix $u\in  \ZZ_{>0}$, $v\in \NN$, and $\ba\in u\Gamma\setminus (u+1)\Gamma$ with $\ba+v\fe_j\in \cone_u(F)$.  
			Consider the function 
			\[g:\RR_{\gs 0}\to \RR_{>0},\qquad g(t)=\min\{\ell_{G}(\ba+t\fe_j)\mid G\in \F_{d-1}(\Gamma)\}. 
			\]
			Since $\fe_j$ is an unbounded direction of $F$, we have 
			\[
			\ell_F(\ba+t\fe_j)=\ell_F(\ba)=\ell_F(\ba+v\fe_j)\in [u,u+1)
			\quad \text{ for every }t\gs 0
			\] 
			Moreover,  
			for every  $t\gs v$ we have $\ba+t\fe_j \in \cone_u(F)$ and then,  by   \autoref{l:uGamma}(1),    $g(t)=\ell_F(\ba)$.

			Let $0\ls t<v$. 
			If $g(t)=\ell_F(\ba)$, then $\ba+t\fe_j\in  \RR_{>0}F$ by \autoref{l:uGamma}(1). 
			Thus, $\ba+t\fe_j\in \cone_u(F)$.  If $g(t)\neq \ell_F(\ba)=g(v)$, then $\ba+t\fe_j$ is not in $\cone_u(G)$ for any $G\in \F_{d-1}(\Gamma)$ with  $\fe_j$  unbounded, as otherwise  $\ba+v\fe_j\in \cone_u(G)$ and so $g(t)=\ell_G(\ba)=g(v)$ by \autoref{l:uGamma}(1), a contradiction. 
			
			Let $F_1,\ldots, F_r$ be the elements of $\F_{d-1}(\Gamma)$ for which  $\be_j$  is bounded. 
			From the  analysis above and \autoref{l:uGamma}(2), we conclude   $\ba+t\be_j\in \cone_u(F)\cup\bigcup_{i=1}^r\cone_u(F_i)$ for every $t\in \RR_{\gs 0}$. 
			Thus, setting $l_{F,j}:=\max\{0, t_{F_1,j},\ldots, t_{F_r,j}\}$ with $t_{F_i,j}$  as in (1), 
			we have  $\ba+l_{F,j}\,\be_j\in \cone_u(F)$, finishing the proof. 
		\end{proof}

		The following auxiliary lemma helps us simplify the computations in the proof of the next proposition. The definition of the function $\L_F(u,v)$ was given in \autoref{eq:Lfunction}.

		\begin{lemma}\label{l:super_pol}
			Let $\Delta=\RR_{\gs 1}\conv(\fe_1,\ldots, \fe_d)$. For every $F\in \F_{d-1}(\Gamma)$  we have 
			\[
			\#\left[
			\Lat{\cone_u(F)}
			\setminus 
			\left(
			\Lat{u\Gamma\setminus(u+1)\Gamma} + v\Delta
			\right)
			\right]=\L_F(u,v)+o(d-1).
			\]
		\end{lemma}
		\begin{proof}
			If $F$ is bounded, by \autoref{l:bounded_dir_goes_out}(1) for  
			$v\gs t_{F,1}+\ldots + t_{F,d}$ we have
			\[
			\#\left[
			\Lat{\cone_u(F)}
			\setminus 
			\left(
			\Lat{u\Gamma\setminus(u+1)\Gamma} + v\Delta
			\right)
			\right]
			=
			\#\Lat{\cone_u(F)}
			= \L_F(u,v),
			\]	
			and so the conclusion follows. 
			
			For the rest of the proof we assume $F$ is unbounded. 
			Since $\Delta_F^{\gs}\subseteq \Delta$ and $\cone_u(F)\subseteq u\Gamma\setminus (u+1)\Gamma$, from \autoref{l:uGamma}(2), it follows that 
			\[
			\#\left[
			\Lat{\cone_u(F)}
			\setminus 
			\left(
			\Lat{u\Gamma\setminus(u+1)\Gamma} + v\Delta
			\right)
			\right]\ls \L_F(u,v).
			\]
			
			By \autoref{p:EG_pol} and \autoref{rem:complexitydown} it suffices  to show that there exists $c\in \NN$ such that for every $v\gs c$ we have 
			\[
			\#\left[
			\Lat{\cone_u(F)}
			\setminus 
			\left(
			\Lat{u\Gamma\setminus(u+1)\Gamma} + v\Delta
			\right)
			\right]\gs \L_F(u,v-c).
			\]
			Assume without loss of generality that the unbounded directions of $F$ are  $\be_{r+1},\ldots \be_{d}$. Let 
			$\ba\in \Lat{u\Gamma\setminus(u+1)\Gamma}$  and $\bb=(b_1,\ldots, b_d)\in \Latt{v\Delta}$ be such that $\ba+\bb\in \cone_{u}(F)$. By \autoref{l:bounded_dir_goes_out}(1) we have $b_i< t_{F,i}$ for every $1\ls i\ls r$.
			Moreover, by \autoref{l:bounded_dir_goes_out}(2), setting  $l=\max\{l_{F,r+1},\ldots,  l_{F,d}\}$ and 
			$\bb'=\bigl(b_1,\ldots, b_r,\min\{l,b_{r+1}\}, \cdots, \min\{l,b_{d}\}\bigr)$, we have  $\ba+\bb'\in \cone_{u}(F)$. 
			
			Set $c= t_{F,1}+\ldots + t_{F,r}+(d-r)l$ and assume $v> c$. 
			We observe that $\bb-\bb'\in (v-c)\Delta_F^{\gs}$ and so $\ba+\bb=\ba+\bb'+(\bb-\bb')\in \cone_{u}(F)+(v-c)\Delta_F^\gs$. Thus,
			\[
			\Lat{\cone_u(F)}
			\setminus 
			\left(
			\Lat{\cone_u(F)} + (v-c)\Delta_F^{\gs }
			\right)
			\subseteq
			\Lat{\cone_u(F)}
			\setminus 
			\left(
			\Lat{u\Gamma\setminus(u+1)\Gamma} + v\Delta
			\right),
			\] 
			finishing the proof.
		\end{proof}

		
		\begin{corollary}\label{c:coordinateBound}
			Let $\Delta=\RR_{\gs 1}\conv(\fe_1,\ldots, \fe_d)$. For each $1\ls i\ls d$ let $H_{i,u}$ be the intersection of $u\Gamma\setminus (u+1)\Gamma$ with the coordinate plane $\{x_i=0\}$. Then
			\[
			\C_i(u,v):= 
			\#\left[
			\Lat{H_{i,u}}
			\setminus 
			\left(
			\Lat{u\Gamma\setminus(u+1)\Gamma} + v\Delta
			\right)
			\right]=o(d-1).
			\]
		\end{corollary}
		\begin{proof}
			We may safely assume $\C_i(u,v)\neq o(0)$, as otherwise the result is trivial. 
			
			By \autoref{l:uGamma}(2), \autoref{l:super_pol}, and \autoref{p:EG_pol}, we have 
			\begin{equation}\label{eq:od-1}
				\#\left[
				\Lat{u\Gamma\setminus(u+1)\Gamma}
				\setminus 
				\left(
				\Lat{u\Gamma\setminus(u+1)\Gamma} + v\Delta
				\right)
				\right]=o(d).
			\end{equation}
			The result then follows by noticing that $\C_i(u,v)$ is equal to the analogous lattice counting in \autoref{eq:od-1} corresponding to the monomial ideal 
			$I\cap \kk[x_1,\ldots,x_{i-1},x_{i+1},\ldots, x_d]_{(x_1,\ldots,x_{i-1},x_{i+1},\ldots, x_d)}$.
		\end{proof}
		
		We are now ready to prove the next proposition.

		\begin{proposition}\label{p:Inequality}
			Let $\fm = (x_1,\ldots, x_d)$ be the maximal ideal of $\kk[x_1,\ldots,x_d]_{(x_1,\ldots, x_d)}$. 	For  $u,v\in \ZZ_{>0}$ we have
			\[
			\lambda_R\left(\frac{\overline{I^u}}{\fm^{v} \overline{I^u}+\overline{I^{u+1}}}\right)
			=
			\sum_{F\in\F_{d-1}(\Gamma)}
			\L_F(u,v)
			+ o(d-1).
			\]
		\end{proposition}
		\begin{proof}
			Let $\Delta=\RR_{\gs 1}\conv(\fe_1,\ldots, \fe_d)$ be the Newton polyhedron of $\fm$. 
			For  $u,v\in \ZZ_{>0}$ and  a subset $A\subseteq u\Gamma$ we use the notation 
			\[
			\D\bigl(A\bigr) = 
			\#\Lat{A}
			\setminus 
			\left(
			\Lat{u\Gamma\setminus(u+1)\Gamma} + v\Delta
			\right).
			\]
			For any $u\in \ZZ_{>0}$ the ideal  $\overline{I^u}$ is monomial and its set of  monomials  corresponds to $\Latt{u\Gamma}$, see  \cite[Propositions~1.4.2~and~1.4.6]{HunekeSwansonIntegral}. 
			Therefore, 
			\[
			\lambda_R\left(\frac{\overline{I^u}}{\fm^{v} \overline{I^u}+\overline{I^{u+1}}}\right)
			=
			\D\bigl(u\Gamma\setminus(u+1)\Gamma\bigr).
			\]
			Moreover, if $N$ is the cardinality of $\F_{ d-1}(\Gamma)$ and $\C_i(u,v)$ is as in \autoref{c:coordinateBound}, then 
			\begin{align*}
				0\ls	
				\sum_{F\in \F_{ d-1}(\Gamma)}
				\D\bigl(\cone_u(F)\bigr)
				-
				\D\bigl( u\Gamma\setminus(u+1)\Gamma\bigr)
				&\ls
				N\sum_{G\in \F_{ \ls d-2}(\Gamma)} 
				\D\bigl(\cone_u(G) \bigr)+ N\sum_{i=1}^d\C_i(u,v)\\
				&\ls N\sum_{G\in \F_{\ls d-2}(\Gamma)}  \L_G(u,v)+ N\sum_{i=1}^d\C_i(u,v),
			\end{align*}
			where the first  inequality follows from  
			\autoref{l:uGamma}(2) and the third one from the containments 
			$\Delta_G^\gs\subset \Delta$ and $\cone_{u}(G)\subset u\Gamma\setminus (u+1)\Gamma$. It remains to justify the second inequality, which we do below; the proof then follows from \autoref{p:EG_pol}, \autoref{l:super_pol}, and \autoref{c:coordinateBound}.
			
			Any point $\ba \in 
			\Lat{u\Gamma\setminus(u+1)\Gamma}
			\setminus 
			\left(
			\Lat{u\Gamma\setminus(u+1)\Gamma} + v\Delta
			\right) $ is counted on the left hand side  of the second inequality at most $N-1$ times, and it is not counted when it belongs to exactly one $\cone_u(F)$. Therefore, it suffices to show that if a point  belongs to more than one $	\cone_u(F)$ and it is not in any of the coordinate planes, then it belongs to at least one  $\cone_u(G)$ with  $G\in \F_{\ls d-2}(\Gamma)$. 
			
			Suppose $\ba\in \cone_u(F)\cap \cone_u(F')$ for different $F,F'\in \F_{d-1}(\Gamma)$. Then, by \autoref{l:uGamma}(1), $s:=\ell_F(\ba)=\ell_{F'}(\ba)\in [u,u+1)$, and $\ell_{F''}(\ba)\gs s$ for any other $F''\in \F_{d-1}(\Gamma)$. Therefore, $s^{-1}\ba\in F\cap F'$. If  $\ba$ is not in any of the coordinate planes, then it belongs to the non-coordinate face $F\cap F'$, which necessarily is of dimension at most $d-2$, finishing the proof. 
		\end{proof}
		
		The following is the main result of this section (and of this paper). We follow the notation introduced at the start of this section. 
		
		\begin{theorem}\label{thm:main}
			Let $I\subsetneq \kk[x_1,\ldots,x_d]_{(x_1,\ldots, x_d)}$ be a nonzero monomial ideal and let  $\Gamma$ be its Newton polyhedron.  
			For every $0\ls i\ls d$ we have 
			\[
			c_i(I)=
			\sum_{F\in\F_{d-1}(\Gamma)}d!\,\Bigl(\MV_d\bigl(\pyr(F)[d-i],\,\Delta_F^{\leqslant}\,[i]\bigr)-
			\MV_d\bigl(\pyr(F)[d-i],\,\Delta_F[i]\bigr)
			\Bigr).
			\]
		\end{theorem}
		\begin{proof} 
			For $i=d$ both terms in the equality are zero by \autoref{rem:cd} and \autoref{rem:mixed_vol}. 
			The equality for $0\ls i\ls d-1$ follows from \autoref{p:Inequality}, \autoref{p:EG_pol}, \autoref{l:passing_to_int_closure}, and \autoref{thm:mixed_Ehrhart_structure}. 
		\end{proof}

		\begin{remark}\label{rem:main_at_c_0}
			We have $\MV_d\bigl(\pyr(F)[d],\,\Delta_F^{\leqslant}\,[0]\bigr) =
			\MV_d\bigl(\pyr(F)[d],\,\Delta_F[0]\bigr)$ whenever $\Delta_F\neq \emptyset$; see \autoref{rem:empty_zero_MV}. Therefore, from \autoref{thm:main}, \autoref{eq:conv_bounded}, and  \autoref{rem:mixed_vol} we obtain that   $c_0(I)$ is the sum of $d!\Vol_d\bigl(\pyr(F)\bigr)$ over the bounded facets of $\Gamma$, thus recovering the main result of  \cite{JMULT_MONOMIAL}. 
		\end{remark}
		
		In the following example we use \autoref{thm:main} to compute the multiplicity sequence of the monomial ideal in \autoref{e:notation_cones}.
		
		\begin{example}\label{e:main}
			Let $I$ be the monomial ideal from \autoref{e:notation_cones}. We begin by computing the mixed volumes of each $\pyr(F_i)$ with the corresponding $\Delta_{F_i}^{\ls}$.
			
			$F_1$: The solid $u\pyr(F_1)+v\Delta_{F_1}^{\ls}$ has vertices \[
			\{(0,0,0),(2u,0,u),(u,u,0),(v,0,0),(2u+v,0,u),(u+v,u,0)\}
			\]
			and volume  $\frac{1}{2}u^2v$. Therefore, the only nonzero mixed volume is 
			$\MV_3\bigl(\pyr(F_1)[2],\,\Delta_{F_1}^{\leqslant}\,[1]\bigr)=\frac16$.
			
			$F_2$: The solid $u\pyr(F_2)+v\Delta_{F_2}^{\ls}$ has vertices \[
			\{(0,0,0),(2u,0,u),(u,u,0),(0,2u,u),(0,0,v),(2u,0,u+v),(0,2u,u+v)\}
			\]
			and volume $\frac23 u^3+2u^2v$. Therefore, the nonzero mixed volumes are 
			$\MV_3\bigl(\pyr(F_2)[3],\,\Delta_{F_2}^{\leqslant}\,[0]\bigr)=
			\MV_3\bigl(\pyr(F_2)[2],\,\Delta_{F_2}^{\leqslant}\,[1]\bigr)=\frac23$.
			
			$F_3$: The computation is the same as with $F_1$ up to exchanging the roles of $x,y$. Therefore, the only nonzero mixed volume is 
			$\MV_3\bigl(\pyr(F_3)[2],\,\Delta_{F_3}^{\leqslant}\,[1]\bigr)=\frac16$.

			Among the solids 
			$u\pyr(F_i)+v\Delta_{F_i}$, the only full-dimensional one is the corresponding to $i=2$. This solid has vertices 
			\[
			\{(0,0,v),(2u,0,u+v),(u,u,v),(0,2u,u+v)\}
			\]
			and  volume  $\frac23 u^3$. Therefore, the only  nonzero mixed volume is  
			$\MV_3\bigl(\pyr(F_2)[3],\,\Delta_{F_2}[0]\bigr)=\frac23$.
			
			Putting it all together in \autoref{thm:main} we obtain $(c_0(I),c_1(I),c_2(I),c_3(I))=(0,6,0,0)$.
			
		\end{example}

		\begin{remark}\label{rem:M2_and_comp}
			In low dimensions, mixed volumes can often be computed directly, as in \autoref{e:main}. In general, it is more practical to use computer  tools, such as the Macaulay2 packages \texttt{Polyhedra} and \texttt{MixedMultiplicity} \cite{Birkner09,GoelMukundanRoyVerma23,MACAULAY2}. We also note that multiplicity sequences can be computed in Macaulay2 using the package \texttt{MultiplicitySequence} \cite{ChenKimMontano24}.	
		\end{remark}

		\section{Counterexample to Achilles-Manaresi's conjecture}\label{s:counterexample}
		
		In order to state the conjecture  we need to introduce some additional terminology. We continue with the notation from \autoref{s:mult-seq}.
		
		For each $0\ls k\ls d-1$, let $\mathscr{F}(k)$ denote the set of facets $F$ of $\Gamma$ (bounded or unbounded)  satisfying the following:
		\begin{enumerate}
			\item[(i)] $F$ has at least $d-(k+1)$ unbounded directions, and 
			\item[(ii)] $F$ contains  at least one bounded face of dimension $k$.
		\end{enumerate}
		We note that $\mathscr{F}(d-1)$ is the set of bounded facets. 
		
		To each pair $F\in \mathscr{F}(k)$ and $k$-dimensional bounded face $G\subseteq F$ we define $\Vol(G,F)$ as follows. 
		Let $M_1,\ldots, M_q$ be the sets of  $d-(k+1)$ unbounded directions of $F$ for which the projection  
		\[
		\pi_i:\RR^d\to \sum_{j\not\in M_i}\RR\fe_j
		\]
		sends $G$ to a $k$-dimensional convex set $\pi_i(G)$. Define
		\[
		\Vol(G,F) = (k+1)!\min_{1\ls i\ls q}\bigl\{\Vol_{k+1}\bigl(\conv(\0,\pi_i(G))\bigr)\bigr\}.
		\]
		
		\begin{conjecture}[{Achilles-Manaresi, \cite[Conjecture 1]{AchillesManaresi2022}}]\label{conj:AM}
			For every $0\ls k\ls d-1$ we have
			\[
			c_{d-(k+1)}(I)=
			\sum_{F\in \mathscr{F}(k)}\min_{\substack{
					\scriptscriptstyle G\subseteq F\, \text{\rm bounded}\\
					\scriptscriptstyle \dim(G)=k}}\bigl\{\Vol(G,F)\bigr\}
			\]
		\end{conjecture}
		
		The conjecture holds for $k=d-1$ by \cite[Theorem 3.2]{JMULT_MONOMIAL} and for $k=0$ by  \cite[(3--5)~and~Theorem~2]{AchillesManaresi2022}. Moreover, $c_d(I)$ always vanishes  in our setting; see \autoref{rem:cd}. Therefore, \autoref{conj:AM} holds up to dimension two. In our next theorem we show that the conjecture fails in dimension three.
		
		\begin{theorem}\label{thm:AM-conj-false}
			\autoref{conj:AM} fails for  $I=(x_1^2x_3,x_2^2x_3,x_1x_2)\subset \kk[x_1,x_2,x_3]_{(x_1,x_2,x_3)}$,  the monomial ideal from \autoref{e:notation_cones}. In particular, the conjecture does not hold in dimension three.
		\end{theorem}
		\begin{proof}
			Following the notation in \autoref{e:notation_cones}, we obtain $\mathscr{F}(1)=\{F_1,F_2,F_3\}$. The only one-dimensional bounded face of $F_1$ is $G_1=\conv(\fv_1, \fv_3)$, and similarly 
			$G_2=\conv(\fv_2, \fv_3)$ is the only one-dimensional face of $F_3$. On the other hand, $F_2$ has two one-dimensional faces, $G_1$ and $G_2$. We compute 
			\[
			\Vol(G_1,F_1) = 1,\qquad
			\Vol(G_2,F_3) = 1,\qquad
			\Vol(G_1,F_2) = \Vol(G_2,F_2)=2. 
			\]
			Therefore, \autoref{conj:AM} predicts $c_1(I)=4$. However, $c_1(I)=6$ as computed in \autoref{e:main}. 
		\end{proof}

		\section{Mixed multiplicities of monomial ideals}\label{s:mixed_mult}
		
		We continue with the notation from the previous sections.  
		%
		In our next theorem, we give a mixed-volume for the mixed multiplicities of arbitrary monomial ideals. 
		\begin{theorem}\label{thm:mixed_mult}
			Let $J, I_1,\ldots, I_r\subseteq \kk[x_1,\ldots,x_d]_{(x_1,\ldots, x_d)}$ be nonzero monomial ideals with $J$ $\fm$-primary. Let $Q,P_1,\ldots, P_r$ be the convex hulls of the vertices of the Newton polyhedra of $J, I_1,\ldots, I_r$, respectively. 
			For every $n_0,\ldots, n_r\in \NN$ with $n_0+\ldots + n_r=d-1$ the mixed multiplicity $e_{(n_0,\ldots, n_r)}(J|I_1,\ldots I_r)$ is equal to
			\[ 
			d!\Bigl(
			\MV_d\bigl(
			\conv(\0,Q)[n_0+1], P_1[n_1],\ldots, P_r[n_r]
			\bigr)-
			\MV_d\bigl(
			Q[n_0+1], P_1[n_1],\ldots, P_r[n_r]
			\bigr)
			\Bigr).
			\] 
		\end{theorem}
		\begin{proof}
			Let $\Lambda=Q+\RR_{\gs 0}^d$ and for each $1\ls i\ls r$ let $\Gamma_i=P_i+\RR_{\gs 0}^d$. 
			We note that 
			\[
			\M(u_0,\ldots, u_r):=
			\#\Latt{
				(u_1\Gamma_1+\cdots+u_r\Gamma_r)\setminus 
				(u_0\Lambda+u_1\Gamma_1+\cdots+u_r\Gamma_r)
			}=
			\lambda_R\left(\frac{\overline{I_1^{u_1}\cdots I_r^{u_r}}}{\overline{J^{u_0}I_1^{u_1}\cdots I_r^{u_r}}}\right).
			\]
			Proceeding as in \autoref{p:passing_to_Ehrhart} via \autoref{l:convex_stuff} we obtain
			\[
			\M(u_0,\ldots, u_r)=
			E_{\conv(\0,Q),P_1,\ldots, P_r}(u_0,\ldots, u_r)-
			E_{Q,P_1,\ldots, P_r}(u_0,\ldots, u_r).
			\]
			We also note that 
			\[
			\lambda_R\left(\frac{\overline{J^{u_0}I_1^{u_1}\cdots I_r^{u_r}}}{\overline{J^{u_0+1}I_1^{u_1}\cdots I_r^{u_r}}}\right)
			= \M(u_0+1,\ldots, u_r)-\M(u_0,\ldots, u_r).
			\]
			The result now follows  by \autoref{thm:mixed_Ehrhart_structure} and \autoref{l:passing_to_int_closure_mixed_mult}.
		\end{proof}
		
		\begin{remark}\label{rem:mixed_non_ic}
			In \cite[Example 6.9]{CidRPolUlr24}, the authors give an example showing that, unlike the multiplicity sequence, mixed multiplicities do not characterize integral dependence. \autoref{thm:main} and \autoref{thm:mixed_mult} shed some light on this distinction. Indeed, for monomial ideals, the multiplicity sequence uses information from all the  facets of the Newton polyhedron, whereas mixed multiplicities only use the bounded facets.
		\end{remark}

		\begin{center}
			{\it Acknowledgments}
		\end{center}
		
		The second author is grateful to Federico Castillo for helpful discussions about mixed volumes of polytopes.  The second author was partially funded by NSF Grant  DMS \#2401522.  
		ChatGPT and Claude were used to find relevant references and to improve the readability of some sentences throughout the article.

		\bibliographystyle{plain}

		\bibliography{References}
		
		
	\end{document}